\documentclass[reqno,A4paper]{amsart}

\usepackage{amsmath, amsthm}\usepackage{enumerate}\usepackage{amssymb}\usepackage{bbold}\usepackage{color}
\usepackage{bbm}\usepackage{dsfont}\usepackage{color}\usepackage{xcolor}
\usepackage{mathrsfs} \usepackage{xcolor}
\usepackage{eqlist}
\usepackage{array}

\theoremstyle{plain}
\newtheorem{theorem}{Theorem}[section]
\newtheorem{proposition}[theorem]{Proposition}
\newtheorem{lemma}[theorem]{Lemma}
\newtheorem{corollary}[theorem]{Corollary}

\theoremstyle{definition}
\newtheorem{definition}[theorem]{Definition}
\newtheorem{question}[theorem]{Question}
\newtheorem{remark}[theorem]{\textup{Remark}} 
\newtheorem{example}[theorem]{\textit{Example}} 

\numberwithin{equation}{section}

\DeclareMathOperator{\stc}{\xrightarrow[]{st}}

\DeclareMathOperator{\nc}{\xrightarrow[]{\rVert\cdot\rVert}}
\begin{document}
\title[]%
{Introducing Statistical Operators: Boundedness, Continuity, and Compactness}
\author[E. Bayram, M. Küçükaslan, M. Et, A. Ayd\i n]
{Erdal Bayram$^{1}$ \and Mehmet Küçükaslan$^2$ \and Mikail Et$^3$ \and Abdullah Ayd\i n$^{4,*}$}

\newcommand{\acr}{\newline\indent}

\address{\llap{1\,}
	Department of Mathematics\acr
	Tekirdağ Nam\i k Kemal University\acr
	Tekirdağ, TURKEY}
\email{ebayram@nku.edu.tr}

\address{\llap{2\,}Department of Mathematics\acr
	Mersin University\acr
	Mersin, TURKEY}
\email{mkkaslan@gmail.com}

\address{\llap{3\,}Department of Mathematics\acr
	F\i rat University\acr
	Elaz\i\c{g}, TURKEY}
\email{met@firat.edu.tr }

\address{\llap{4\,} Department of Mathematics\acr
	Mu\c{s} Alparslan University\acr
	Mu\c{s}, TURKEY}
\email{$^*$Correspondence:a.aydin@alparslan.edu.tr}

\begin{abstract}
Many studies have been conducted on statistical convergence, and it remains an area of active research. Since its introduction, statistical convergence has found applications many fields. Nevertheless, there is a shortage of research related to operator theory. As far as we know, no studies have focused on continuous, bounded, and compact operators, which are fundamental concepts in mathematics. We explore the notions of statistical boundedness, continuity, and compactness of operators between normed spaces, establishing connections between these concepts and their counterparts in traditional normed space theory. Additionally, we provide examples and results that demonstrate the behavior and implications of statistical convergence in the context of operators.
\end{abstract}
	
\maketitle
\section{Introduction and Preliminaries}\label{Sec:1}
The boundedness, continuity and compactness of operators are crucial in understanding the principles of functional analysis theory, and also theses concepts constitute fundamental research in many fields of mathematical analysis. These notions are often associated with convergences that are either topological or non-topological because many deep properties are often expressed in terms of convergent nets and sequences. Different types of convergence are distinguished on the basis of the underlying mathematical structure, and comparisons between convergence types defined on the same structure can be performed. Statistical convergence is one type of convergence that has received a lot of attention lately. Statistical convergence emerged for real sequences as an extension of the conventional topological convergence, firstly introduced by by Steinhaus \cite{steinhaus} and by Fast \cite{fast} separately, and then generalised by Schoenberg \cite{schoenberg}. Salat \cite{salat}, Fridy \cite{fridy}, and Connor \cite{connor} have made substantial contributions by clarifying important characteristics of statistical convergence for real sequences. Furthermore, Maddox \cite{Maddox} provides evidence of applications that illustrate its practical utility, and also Fridy and Orhan obtain main results of statistical limit superior and limit inferior\cite{FO}. Tripathy \cite{Tripathy}, Alt\i nok et al. \cite{Altinok}, Bhardwaj and Gupta \cite{BG}, and Temizsu and Et \cite{TE} obtain some generalizations and properties of statistical bounded sequences. Ayd\i n introduce some results of statistical order convergence on Riesz spaces by order convergence  \cite{Aydn,AE}, Akbaş and Işık give some results in probability \cite{AI}, Baliarsingh et al. get results on deferred summable functions \cite{Et}, and Yapali and Polat get results on fuzzy normed spaces \cite{Reha}. Despite the considerable attention statistical convergence has received from numerous authors since its inception, leading to numerous applications and generalizations of the concept, there are a few studies concerning the concept of operators directly related to statistical convergence. Specifically, as far as we know, there is no any work about statistical continuity, boundedness, and compactness of operators. Hence, in the current paper, our goal is to present the concept of operators associated with statistical convergence.

Before continuing with the presentation of results, it would be advantageous for the reader to revisit the definitions and terminology employed in this study. Let us now remind some fundamental properties of concepts associated with statistical convergence. The \textit{natural density} (the main tool of statistical convergence) of a subset $K\subseteq\mathbb{N}$ is defined as follows:
$$
\delta(K) := \lim_{n \rightarrow \infty} \frac{1}{n} \left\vert \{ k \leq n : k \in K \} \right\vert
$$
Note that the vertical bars represent the cardinality of the specified set in the definition of statistical convergence. For further exposition on the natural density of sets, we refer the reader to \cite{fast, fridy}. Similarly, a sequence $(x_n)$ in a normed space $\left( U,\|\cdot\|\right) $ is called \textit{statistically convergent} to $x\in U$ if we have the existence of the following limit
$$
\lim_{n \rightarrow \infty} \frac{1}{n}\left\vert\{k\leq n:\|x_{k}-x\|\geq\varepsilon\}\right\vert=0
$$
for each $\varepsilon > 0$. It is well known that the statistical limit of a sequence is unique on normed spaces whenever it is exits. On the other hand, to simplify matters, we establish the following convention: if $(x_n)$ is a sequence that adheres to property $Q$ for all $n$ except for a negligible subset with a natural density of zero, then we denote that $(x_n)$ satisfies $Q$ for almost all $n$, abbreviated as a.a. $n$.

The organization of this paper is outlined as follows: Section $2$ presents an introduction to various concepts of statistical bounded maps between normed spaces, exploring their relationships and connections with other established operators. Section $3$ establishes the definition of statistical continuous maps through the utilization of statistical convergence on normed spaces. Section $4$ elaborates on the methodology for defining statistical compact operators. In Section $5$, we present the concept of statistical complete normed spaces, defining the notion of a statistical Banach space.

\section{Statistical bounded operators}\label{Sec:2}
The concept of statistical boundedness formally introduced by Fridy and Orhan \cite{FO} as follows: a sequence of real numbers, represented as $x := (x_n)$, is considered statistically bounded if we have $\delta\left(\{k : |x_k| > M\}\right)=0$ for some positive numbers $M$. It is established that statistical convergent and statistical Cauchy sequences are statistically bounded. However, it is remarkable that a statistically bounded sequence may not necessarily be statistically convergent or statistical Cauchy; this is demonstrated in, for instance, Example 2.10 \cite{BB}. In a similar vein, statistically bounded sequence in normed spaces is defined as follows:
\begin{definition}
	A sequence $x := (x_n)$ in a normed space $U$ is considered {\em statistically norm bounded} if there exists a positive real number $M > 0$ such that the expression
	$$
	\delta\left(\left\{n \in \mathbb{N} : \left\Vert x_{n} \right\Vert > M \right\}\right) = 0
	$$
	holds true.
\end{definition}

Alternatively, a sequence $x:= (x_n)$ is considered statistically bounded if $\|x_n\| \leq M$ for almost all $n$ satisfies for some positive real numbers $M > 0$ such that  (cf. \cite{Altinok,BG,TE,Tripathy}). It has been established that any statistical Cauchy sequence is statistically bounded (cf. \cite[Thm.8]{Altinok}). However, it is important to note that a statistically bounded sequence may not necessarily be statistically convergent or statistical Cauchy. Additionally, it has been shown in Example 5 \cite{Altinok} that a statistically bounded sequence may not be bounded in norm. Lemma 2.8 and Corollary 2.9 \cite{BB} indicate that weakly statistically Cauchy and weakly statistically convergent sequences in a normed space are statistically bounded; however, the converse does not hold in general. Throughout this paper, the set $\ell_{\infty}(U)$ represents all norm-bounded sequences in the normed space $U$, while the set containing all statistically bounded sequences in $U$ is symbolized by $\ell_{\infty}^{st}(U)$. It is clear that $\ell_{\infty}(U) \subset \ell_{\infty}^{st}(U)$.

Recall that a sequence $x=(x_n)$ is said to be {\em weakly statistical convergent} in a normed space $U$ if $(f(x_n))$ is statistical convergent in $\mathbb{R}$ for all $f\in U'$; see \cite[Def.1.1(ii)]{BB}, and also it is called {\em weakly statistical bounded} in $U$ if $(f(x_n))$ is statistical bounded in $\mathbb{R}$ for all $f\in U'$, where $U'$ denotes the norm dual of $U$.
\begin{theorem}\label{weak and st bounded coincide}
	The concepts of weakly statistically bounded and statistically bounded sequences in normed spaces are equivalent.
\end{theorem}

\begin{proof}
	Suppose \(x := (x_n)\) is regarded as a weakly statistically bounded sequence within a normed space \(U\), meaning that for every \(f \in U'\), the sequence \((f(x_n))\) is statistically bounded in \(\mathbb{R}\). This implies that, for any \(f \in U'\), there exists a subsequence \((x_{k_n})_{k_n \in K}\) of \((x_n)\) with \(\delta(K) = 1\) such that \((f(x_{k_n}))\) is bounded conventionally in \(\mathbb{R}\). In other words, there exists a scalar \(M_f\) such that \(f(x_{k_n}) \leq M_f\) holds for all \(k_n \in K\). Now, let's consider the canonical map \(\Phi\) from \(U\) to \(U''\), defined as \(\Phi(x) = h_x\) for all \(x \in U\) and \(h_x(f) = f(x)\) for all \(f \in U'\), where \(\|h_x\| = \|x\|\). Choose any \(f \in U'\). Then, we can say:
	$$
	\sup_n |h_{x_{k_n}}(f)| = \sup_n |f(x_{k_n})| < M_f.
	$$
	In accordance with the Banach-Steinhaus theorem, $\sup_n \|h_{x_{k_n}}\|$ is bounded, thus we obtain that $\sup_n \|x_{k_n}\|$ is also bounded because the dual space $U'$ is a Banach space. Therefore, we conclude the $st$-boundedness of $(x_n)$.
	
	Conversely, let $(x_n)$ be a statistically bounded sequence in a normed space $U$. This implies the existence of a positive real number $M$ such that:
	$$
	\delta(\{n \in \mathbb{N}: \|x_n\| \leq M\}) = 1.
	$$
	Since $f \in U'$ is both continuous and linear, there exists a constant $K$ such that $|f(x)| \leq K\|x\|$ for all $x \in U$. Hence, we have
	$$
	\delta(\{n \in \mathbb{N}: |f(x_n)| \leq K\|x_n\| \leq KM\}) = 1.
	$$
	This demonstrates that $(f(x_n))$ is statistically bounded for any continuous linear functional $f$, which implies that $(x_n)$ is a weakly statistically bounded sequence.
\end{proof}

It is important to clarify the notation used in this paper. We denote by $L(U,V)$ the set of all operators (functions defined between vector spaces) and by $\mathcal{L}(U,V)$ the set of all linear operators between normed spaces $U$ and $V$. Throughout this paper, unless explicitly stated otherwise, $U$ and $V$ are considered as normed spaces. Recall that a norm-bounded operator maps norm-bounded sets to norm-bounded sets, and we denote the set of all such operators by $\mathcal{B}(U,V)$. Building on this, we introduce the following definitions.
\begin{definition}
	An operator $S \in L(U,V)$ is defined as follows:
	\begin{enumerate}
		\item[(1)] It is termed {\em norm-statistically bounded}, abbreviated as {\em $n\text{-}st$-bounded}, if it maps norm-bounded sequences to statistically bounded sequences, i.e., $S(\ell_{\infty}(U)) \subseteq \ell_{\infty}^{st}(V)$.
		\item[(2)] It is referred to as {\em statistically bounded}, denoted as {\em $st$-bounded}, if it maps statistically bounded sequences to statistically bounded sequences, i.e., $S(\ell_{\infty}^{st}(U)) \subseteq \ell_{\infty}^{st}(V)$.
	\end{enumerate}
\end{definition}

The set of all norm-statistical bounded and statistical bounded operators are denoted by $\mathcal{B}_{st}^{n}(U,V)$ and $\mathcal{B}_{st}(U,V)$, respectively. In case $U=V$, $\mathcal{B}_{st}^{n}(U,V)$ and $\mathcal{B}_{st}(U,V)$ are shown by $\mathcal{B}_{st}^{n}(U)$ and $\mathcal{B}_{st}(U)$.   
\begin{lemma}\label{equality}
	The inclusion $\mathcal{B}(U,V) \subseteq \mathcal{B}_{st}(U,V) = \mathcal{B}_{st}^{n}(U,V)$ is established.
\end{lemma}

\begin{proof}
	It is clear that every norm bounded operator is norm-statistical bounded due to the inclusion $\ell_{\infty}\left( U\right)\subseteq\ell_{\infty}^{st}\left( U\right)$. Hence, we obtain $\mathcal{B}(U,V)\subseteq \mathcal{B}_{st}(U,V)$. Moreover, it follows from $S\left(\ell_{\infty}\left( U\right)\right)\subseteq S\left(\ell_{\infty}^{st}\left( U\right)\right)$ that if $S$ is a $st$-bounded operator i.e., $S \left( \ell_{\infty}^{st}\left( U\right)\right) \subseteq \ell_{\infty}^{st}\left( V\right)$ holds, then we have $S\left(\ell_{\infty}\left( U\right)\right)\subseteq S \left( \ell_{\infty}^{st}\left( U\right)\right)\subseteq \ell_{\infty}^{st}\left( V\right)$, and so $S$ is $n$-$st$-bounded. As a result, we have the inclusion $\mathcal{B}_{st}(U,V)\subseteq\mathcal{B}_{st}^{n}(U,V)$.
	
	Assume that $S$ is $n$-$st$-bounded operator and $(x_n)$ is a statistically bounded sequence in $U$. Then, there exists a subsequence $(x_{k_n})_{k_n\in K}$ with $\delta(K)=1$ such that $(x_{k_n})_{k_n\in K}$ is norm bounded in $U$. Thus, by applying the $n$-$st$-boundedness of $S$, we obtain that $(S(x_{k_n}))$ is a $st$-bounded sequence in $V$, and so $(S(x_n))$ is also tatistically bounded sequence in $V$. Hence, we get the following inclusion $\mathcal{B}_{st}^{n}(U,V)\subseteq\mathcal{B}_{st}(U,V)$.
\end{proof}

At the outset of this section, our objective was to examine the principles behind two categories of bounded operators: $st$-bounded and $n$-$st$-bounded. Yet, as indicated by Lemma \ref{equality}, these two sets of operators are equivalent. Hence, our attention is directed solely towards $\mathcal{B}_{st}(U,V)$.

\begin{remark}
	A statistical bounded operator might not necessarily be norm bounded. To illustrate, suppose we have an operator $S:U\to V$ that is statistically bounded, and let $(x_n)$ denote a sequence in $U$ that is norm bounded. Since a norm bounded sequence is also statistical bounded, there exists a scalar $M>0$ and a subsequence $(x_{k_n})_{k_n\in K}$ of $(x_n)$, where $\delta(K)=1$, such that $\|S(x_{k_n})\|\leq M$ for all $k_n\in K$. However, the behavior of $S(x_m)$ for elements $(x_m)_{m\in \mathbb{N}\setminus K}$ remains unspecified. Consequently, $S$ might fail to map bounded sets to bounded sets, thus not qualifying as a norm-bounded operator.
\end{remark}

\begin{example}
	Consider the space $c_{00}$ of all real sequences that eventually terminate in zeros. It is a normed space according to the norm $\left\Vert \left(x_{1},x_{2},\cdots\right) \right\Vert_{c_{00}}=\sup\left\{ \left\vert x_{n}\right\vert :n\in \mathbb{N}\right\}$. Define the operator $S:c_{00}\to c_{00}$ as follows:
	$$
	S(x_{n})= \left \{
	\begin{array}
		[c]{cc}%
		nx_{n}, & n \in \mathbbm{P}\\
		x_{n}, & \text{otherwise}
	\end{array}
	\right.
	$$
	where $\mathbbm{P}$ is the set of all prime numbers. Clearly, $S$ is a linear unbounded operator. On the other hand, for all $n \in \mathbb{N}-\mathbbm{P}$, we have $\left\Vert Sx_{n}\right\Vert_{c_{00}} = \left\Vert x_{n}\right\Vert_{c_{00}}$. That is, for all $st$-bounded sequences $(x_n)$ in the unit ball $B_{c_{00}}$, $\|Sx_n\|_{c_{00}} \leq 1$. This means that $S$ is a $st$-bounded operator.  
\end{example}

\begin{theorem} \label{finite dimension}
	If the normed space $U$ is finite dimensional then every linear operator defined on $U$ is $st$-bounded.
\end{theorem}

\begin{proof} 
	Let $U$ is a finite dimensional normed space with dimension $m$ and $\left\{u_{1},u_{2},...,u_{m}\right\} $ be a Hamel basis of $U$. Consider a linear operator $S:U\rightarrow V$ for any normed space $V$. Suppose $(x_n)$ is a $st$-bounded sequence in $U$. Then, there exists $M>0$ such that $\delta \left( \left\{ n\in	\mathbb{N}:\lVert x_{n}\rVert \geq M\right\} \right) =0$, i.e. $\lVert x_{n}\rVert	\leq M$ holds for almost all $n$. On the other hand, for each $n\in\mathbb{N}$, there exists unique scalars $\alpha _{j}^{n}\in \mathbb{R}$ such that $x_{n}=\sum_{j=1}^{m}\alpha _{j}^{n}u_{j}$. As is commonly understood in normed spaces, there exists $C\in \mathbb{R}_{+}$ such that $\left\Vert x_{n}\right\Vert \geq C\sum_{j=1}^{m}\left\vert\alpha _{j}^{n}\right\vert $ holds. Hence, the inequality $\sum_{j=1}^{m}\left \vert \alpha _{j}^{n}\right\vert \leq \frac{1}{C}M$ holds for almost all $n$. Therefore, we deduce the following inequalities for almost all $n$:
	\begin{equation*}
		\left\Vert Sx_{n}\right\Vert \leq \sum_{j=1}^{m}\left\vert \alpha_{j}^{n}\right\vert \left\Vert Su_{j}\right\Vert \leq \frac{K}{C}M,
	\end{equation*}
	where $K=max\left\{ \left\Vert Su_{j}\right\Vert :j=1,...,m\right\} $. That is, there exists a $L>0$ such that $\left\Vert Sx_{n}\right\Vert \leq L$ for almost all $n$. This means that $S$ is a $st$-bounded operator.
\end{proof}

One might naturally inquire about the conditions under which the equality $\mathcal{B}(U,V)=\mathcal{B}_{\text{st}}(U,V)$ holds. By considering Lemma \ref{equality} and Theorem \ref{finite dimension}, there exists a partial affirmation, as elucidated in the following corollary, applicable to finite-dimensional normed spaces and linear operators.
\begin{corollary}
	If the dimension of $U$ is finite, then $\mathcal{B}(U,V)\cap L(U,V)=\mathcal{B}_{st}(U,V) \cap L(U,V)$.
\end{corollary}

\begin{theorem}\label{M}
	An operator $S\in \mathcal{L}(U,V)$ is $st$-bounded if and only if for each $(x_n)\in\ell_{\infty}^{st}(U)$ there exists a scalar $M>0$ such that $\|S(x_n)\|\leq M\|x_n\|$ for almost all $n$.
\end{theorem}

\begin{proof}
	Assume that $S$ is a $st$-bounded linear operator. For a sequence $(x_n)\in\ell_{\infty}^{st}\left( U\right)$, the sequence $\left(\frac{x_n}{\|x_n\|}\right)$ is also a $st$-bounded sequence. Thus, there exist a scalar $M>0$ such that $\left\lVert S\left(\frac{x_n}{\|x_n\|}\right)\right\lVert\leq M$ for almost all $n$. Therefore, we obtain $\|S(x_n)\|\leq M\|x_n\|$ for almost all $n$.
	
	Conversely, we suppose that for each $x:=(x_n)\in\ell_{\infty}^{st}\left( U\right)$ there exist a scalar $M_{x}>0$ such that $\|S(x_n)\|\leq M\|x_n\|$ for almost all $n$. Since $(x_n)$ is $st$-bounded, we have a scalar $K>0$ such that $x_n\in B(\theta,K)$ for almost all $n$. Hence, we observe the following inequality:
	$$
	\|S(x_n)\|\leq M\|x_n\|\leq MK.
	$$
	for almost all $n$. Therefore, we get $\|S(x_n)\|\leq MK$ for almost all $n$, and so $S$ is a $st$-bounded operator.
\end{proof}

\begin{theorem}
	$\mathcal{B}_{st}\left( U,V\right)$ is a linear subspaces of $L(U,V)$.
\end{theorem}

\begin{proof}
	Assume that $S,T\in \mathcal{B}_{st}\left( U,V\right)$ and $(x_n)$ is a statistical bounded sequence in $U$. Then, there exist positive real numbers $M_{1}$ and $M_{2}$ such that 
	\begin{equation*}
		\delta \left( \left\{ n\in\mathbb{N}:\lVert Sx_{n}\rVert >M_{1}\right\} \right) =0 \ \ \ \text{and} \ \ \ \delta\left( \left\{ n\in\mathbb{N}:\lVert Tx_{n}\rVert >M_{2}\right\} \right)=0.
	\end{equation*}
	On the other hand, for every $n\in\mathbb{N}$, the inequality $\left\Vert \left(S+T\right) x_{n}\right\Vert =\left\Vert Sx_{n}+Tx_{n}\right\Vert \leq \left\Vert Sx_{n}\right\Vert +\left\Vert Tx_{n}\right\Vert$ gives the following inclusion
	\begin{equation*}
		\left\{ n\in\mathbb{N}:\lVert \left( S+T\right) x_{n}\rVert >M_{1}+M_{2}\right\} \subseteq \left\{n:\lVert Sx_{n}\rVert >M_{1}\right\} \cup \left\{n\leq n:\lVert Tx_n\rVert >M_{2}\right\} .
	\end{equation*}%
	Thus, the monotonicity of the natural density implies $\delta \left( \left\{n\in\mathbb{N}:\lVert \left( S+T\right) x_{n}\rVert >M_1+M_2 \right\} \right) =0$, and so $\left( \left( S+T\right) x_{n}\right) $ is a statistical bounded sequence in $V$. Hence, we get $(S+T)\in \mathcal{B}_{st}\left( U,V\right)$.
	
	Now, take any $0\neq\alpha\in\mathbb{R}$. Then, it follows from the equality $\left\Vert \left( \alpha S\right)x_{n}\right\Vert =\left\Vert \alpha Sx_{n}\right\Vert =\left\vert \alpha\right\vert \left\Vert Sx_{n}\right\Vert$ that we have
	\begin{equation*}
		\left\{ n\in\mathbb{N}:\lVert\left(\alpha S\right)x_{n}\rVert>\left\vert\alpha\right\vert M_{1}=M\right\}=\left\{n\in\mathbb{N}:\lVert Sx_{n}\rVert>M_{1}\right\}.
	\end{equation*}
	Thus, we obtain $\delta\left(\left\{n\in\mathbb{N}:\lVert\left(\alpha S\right) x_{n}\rVert>M\right\} \right)=0$, and so $\left(Sx_{n}\right)$ is a statistical bounded sequence in $V$. Therefore, we obtain the desired result $\alpha S\in \mathcal{B}_{st}\left(U,V\right)$.
\end{proof}

\begin{remark}\
	\begin{enumerate}[(i)]
		\item It is well known that every compact operator between normed spaces is norm bounded, and so each compact operator is also $st$-bounded.
		\item Recall that a {\em weakly bounded operator} $S\in L(U,V)$ sends norm-bounded sequences to weakly bounded sequences. Hence, it follows from Theorem \ref{weak and st bounded coincide} and Lemma \ref{equality} that an operator is statistically bounded iff it is weakly bounded.
		\item The composition of $st$-bounded operators is also $st$-bounded. That is, $\mathcal{B}_{st}\left( U\right)$ is a two sided algebraic ideal in itself.
		\item If $S\in\mathcal{B}_{st}(U)$, then $S^{n}\in\mathcal{B}_{st}(U)$ for each $n\in\mathbb{N}_{+}$.
		\item If $S$ is a norm bounded operator and $T$ is a $st$-bounded operator, then their composition $S\circ T$ and $T\circ S$ are also $st$-bounded.
	\end{enumerate}
\end{remark}

Recall that $S:U\to V$ is called \textit{rank one operator} if there exist $f\in U'$ and $y_0\in V$ such that $S(x):=f(x)y_0$ for all $x\in U$. Also, $S$ is called \textit{finite rank operator} if its range is finite dimensional. Accordingly, any finite rank operator is the sum of finite number of rank one operators. 
\begin{theorem}
	A finite rank operator is statistically bounded.
\end{theorem}

\begin{proof}
	Without lost of generality, we suppose that $S$ is given by $S(x)=f(x)y_0$ for some $f\in U'$ and $ y_0 \in V $. Let $(x_n)$ be a $st$-bounded sequence in $U$, then there exists a constant $K>0$ such that
	$$
	\delta(\{n\in \mathbb{N}: \|x_n\|\leq K\})=1.
	$$
	On the other hand, since $f$ is a bounded linear functional, there exists a constant $C>0$ such that $|f(x_n)|\leq C \|x_n\|$ holds for all $n\in\mathbb{N}$. Therefore, the following inequality holds for almost all $n\in\mathbb{N}$:
	$$
	\|S(x_n)\|=\|f(x)y_0\|=|f(x)|\|y_0\|\leq \|x_n\| C\|y_0\|\leq KC\|y_0\|=M
	$$
	Thus, we obtain $\delta(\{n\in \mathbb{N}: \|S(x_n)\|\leq M\})=1$, and so $S$ is a $st$-bounded operator.
\end{proof}

\section{Statistical continuous operators}\label{Sec:3}
The norm continuity of a linear operator between normed spaces is synonymous with the condition that if a sequence $(x_{n})$ converges to the zero vector $\theta$, then the sequence $S(x_{n})$ also converges to $\theta$. It is a widely recognized that the concepts of being norm-bounded and norm-continuous are equivalent for linear operators between normed spaces. Therefore, the collection of all norm-continuous linear operators is identical to $\mathcal{B}(U,V)$. In a similar manner, we introduce the following concepts.
\begin{definition}
	A linear operator $S\in \mathcal{L}(U,V)$ is called 
	\begin{enumerate}
		\item[(1)] {\em norm statistically continuous} (for short, {\em $n$-$st$-continuous}) if $x_n\nc \theta$ implies $S(x_n)\stc \theta$,
		\item[(2)] {\em statistically continuous} (or {\em $st$-continuous}) if $x_{n}\stc \theta$ in $U$ implies $S(x_{n})\stc \theta$ in $V$.
	\end{enumerate}
\end{definition}

In the current manuscript, we represent the collection of all norm continuous linear operators, norm-statistically continuous and statistically continuous operators between normed spaces $U$ and $V$ as $\mathcal{C}(U,V)$, $\mathcal{C}_{st}^{n}(U,V)$ and $\mathcal{C}_{st}(U,V)$, respectively.
\begin{remark}\
	\begin{enumerate}
		\item An isomorphism between normed spaces is $st$-continuous.
		\item The identity operator $I_U\in \mathcal{L}(U)$ is $st$-continuous.
		\item If $S:U\to V$ is a linear operator and $U$ is a finite dimensional normed space, then $S$ is a $st$-continuous operator.
	\end{enumerate}
\end{remark}

\begin{proposition}\label{inclusion}
	The inclusions $\mathcal{C}(U,V)\subseteq \mathcal{C}_{st}(U,V)=\mathcal{C}_{st}^{n}(U,V)$ hold.
\end{proposition}

\begin{proof}
	Let $S\in \mathcal{C}(U,V)$ and $(x_n)$ be a sequence in $U$ such that $x_{n}\stc \theta$. Then, for every $\varepsilon >0$ we have $\delta(\left\{ k\leq n:\left\Vert x_{k}\right\Vert \geq\varepsilon \right\})=0$, and $\left\Vert Sx\right\Vert \leq\left\Vert S\right\Vert \left\Vert x\right\Vert $ holds for every $x\in U$.	Hence, the inequality $\left\Vert Sx_{n}\right\Vert\leq \left\Vert S\right\Vert \left\Vert x_{n}\right\Vert$ holds for each $n\in\mathbb{N}$, which implies the following inclusion
	\begin{equation*}
		\left\{n\in \mathbb{N}:\left\Vert Sx_{n}\right\Vert \geq \left\Vert S\right\Vert \varepsilon\right\} \subseteq \left\{ n\in \mathbb{N}:\left\Vert x_{n}\right\Vert \geq \varepsilon \right\}
	\end{equation*}
	Hence, we have
	\begin{equation*}
		\delta \left( \left\{ n\in\mathbb{N}:\left\Vert Sx_{n}\right\Vert \geq \left\Vert S\right\Vert \varepsilon\right\} \right) \leq\delta\left(\left\{n\in\mathbb{N}:\left\Vert x_{n}\right\Vert \geq \varepsilon \right\}\right)=0
	\end{equation*}
	for every $\varepsilon >0$. This implies that $Sx_{n}\stc \theta$, i.e. we obtain $S\in \mathcal{C}_{st}(U,V)$. Consequently, we have $\mathcal{C}(U,V)\subseteq\mathcal{C}_{st}(U,V)$.
	
	Let $S\in \mathcal{C}_{st}^{n}(U,V)$ and $(x_n)$ be a sequence in $U$ such that $x_{n}\stc \theta$. Then, there exists a subsequence $(x_{k_n})_{k_n\in K}$ with $\delta(K)=1$ such that $(x_{k_n})$ is norm convergent to $\theta$ in $U$. Now, by using the norm statistically continuity of $S$, we obtain $S(x_{k_n})\stc \theta$ in $V$. It means that $(S(x_n))$ is statistical convergent to $\theta$ in $V$. Therefore, $S$ is a statistically continuous operator, and so we have $\mathcal{C}_{st}^{n}(U,V)\subseteq\mathcal{C}_{st}(U,V)$. 
	
	Now, take any operator $S\in \mathcal{C}_{st}(U,V)$ and an arbitrary sequence $(x_n)$ in $U$ such that $\|x_{n}\|\to 0$. Evidently, $(x_{n})$ is statistical covergent to $\theta$. Then, $\left( Sx_{n}\right) $ should be statistical convergent to $\theta$. That is, the operator $S$ sends norm convergent sequences to statistical convergent sequences. This means that $S\in \mathcal{C}_{st}^{n}(U,V)$, and so $\mathcal{C}_{st}(U,V)\subseteq\mathcal{C}_{st}^{n}(U,V)$ holds.
\end{proof}

It is important to recognize that the opposite containment stated in Proposition \ref{inclusion} doesn't hold universally.
\begin{theorem} \label{linearity implies equality}
	A linear operator $S:U\to V$ is $st$-bounded if and only if it is $st$-continuous, i.e. $\mathcal{B}_{st}(U,V) \cap \mathcal{L}(U,V)=\mathcal{C}_{st}(U,V)$.
\end{theorem}

\begin{proof}
	Let $S:U\to V$ be a linear $st$-bounded operator. Consider the sequence $(x_{n})$ in $U$ such that $x_{n}\stc \theta$. Clearly, $(x_{n})$ is $st$-bounded, and so it is follows from Theorem \ref{M} that there exists a scalar $M>0$ such that the inequality
	$$
	\|S(x_n)\|\leq M\|x_n\|
	$$
	holds. Therefore, we have 
	$$
	\left\{ n:\lVert Sx_{n}\rVert \geq \varepsilon \right\} \subseteq  \left\{ n:\lVert x_{n}\rVert \geq \frac{\varepsilon}{M} \right\}
	$$
	for every $\varepsilon >0$. Hence, the fact that $x_{n}\stc \theta$ implies $Sx_{n}\stc \theta$. This means that the operator $S$ is a $st$-continuous operator.     
	
	Conversely, let $S:U\to V$ be a $st$-continuous operator. Suppose that $S$ is not a $st$-bounded operator. Thus, by using Theorem \ref{M}, there exist some sequences $(x_{n})\in \ell_{\infty}^{st}(U)$ such that we have   
	$$
	\delta \left( \left\{ n:\lVert Sx_{n}\rVert>M\lVert x_{n}\rVert \right\}\right) > 0
	$$
	for every $M>0$. By choosing $M=m^2$ for each $m \in \mathbb{N}$, we construct a subsequence $(x_{k_m})$ such that $x_{k_m}$ is a member of $(x_n)$ satisfying $\lVert Sx_{k_{m}}\rVert>m^{2}\lVert x_{k_{m}}\rVert$. Without loss of the generality, we can assume $\lVert x_{k_{m}} \rVert =1$ for all $m$. Otherwise, we can consider the initial sequence $(x_n)$ as $\frac{x_n}{\|x_n\|}$. Define a sequence $(z_{m})=(\frac{1}{m} x_{k_{m}})$. Obviously, $(z_m) \in \ell^{st}_{\infty}(U)$ and $\|z_m\|\to0$. Other hand, for every $m \in \mathbb{N}$, we have
	$$
	\lVert Sz_{m}\rVert=\lVert S(\frac{1}{m} x_{k_{m}})\rVert = \frac{1}{m} \lVert S(x_{k_{m}})\rVert \geq \frac{1}{m} m^2=m.
	$$       
	It follows that $\delta \left( \left\{ k:\lVert Sz_m\rVert \geq \varepsilon \right\} \right) =1$ for every $\varepsilon >0$, i.e. the sequence $(Sz_{k})$ is not $st$-convergent to $\theta$. Hence, $S$ is not $n$-$st$-continuous and so is not $st$-continuous by Theorem \ref{inclusion}, which is a contradiction.  
\end{proof}

\begin{theorem}
	$\mathcal{C}_{st}\left( U,V\right)$ is a linear subspaces of $\mathcal{L}(U,V)$.
\end{theorem}

\begin{proof}
	Let $S,T\in \mathcal{C}_{st}\left( U,V\right)$ and $(x_n)$ be a sequence in $U$ such that $x_{n}\stc x\in U$. Then, for every $\varepsilon >0$, the following facts 
	\begin{equation*}
		\delta \left( \left\{ k\leq n:\lVert Tx_{k}-Tx\rVert \geq \frac{\varepsilon}{2}\right\} \right) =0\text{ \ \ and \ \ }\delta \left( \left\{ k\leq n:\lVert Sx_{k}-Sx\rVert \geq \frac{\varepsilon }{2}\right\} \right) =0.
	\end{equation*}
	provide. On the other hand, for every $n\in\mathbb{N}$, the inequality $\left\Vert \left( S+T\right) x_{n}-\left( S+T\right)x\right\Vert \leq \left\Vert Sx_{n}-Sx\right\Vert +\left\Vert Tx_{n}-Tx\right\Vert $ gives the following inclusion%
	\begin{equation*}
		\left\{ k\leq n:\lVert \left( S+T\right) x_{k}-\left( S+T\right) x\rVert\geq \varepsilon \right\} \subseteq \left\{ k\leq n:\lVert Sx_{k}-Sx\rVert\geq \frac{\varepsilon }{2}\right\} \cup \left\{ k\leq n:\lVert Tx_{k}-Tx\rVert \geq \frac{\varepsilon }{2}\right\} .
	\end{equation*}
	Thus, the monotonicity of the natural density implies $\delta \left( \left\{k\in\mathbb{N}:\lVert \left( S+T\right) x_{k}-\left( S+T\right) x\rVert \geq \varepsilon\right\} \right)=0$, indicating $\left( S+T\right) x_{n}\stc\left(S+T\right) x$ in $V$. Therefore, $S+T\in \mathcal{C}_{st}(U,V)$ holds.
	
	Now, we want to prove $\alpha T\in \mathcal{C}_{st}(U,V)$ for every $\alpha \in \mathbb{R}$. Fix any non zero scalar $\alpha$. Then, we	have $\left\Vert \left( \alpha T\right) x_{n}-\left( \alpha T\right)	x\right\Vert =\left\Vert \alpha Tx_{n}-\left( \alpha T\right) x\right\Vert=\left\vert \alpha \right\vert \left\Vert Tx_{n}-Tx\right\Vert$. Then, we obtain
	\begin{equation*}
		\left\{ k\leq n:\lVert \left( \alpha T\right) x_{n}-\left( \alpha T\right)x\rVert >\left\vert \alpha \right\vert \frac{\varepsilon }{2}\right\}\subseteq \left\{ n\in\mathbb{N}:\lVert Tx_{n}-Tx\rVert >\frac{\varepsilon }{2}\right\}.
	\end{equation*}
	It follows from $\delta \left( \left\{n\in\mathbb{N}:\lVert \left( \alpha T\right) x_{n}-\left( \alpha T\right) x\rVert>\left\vert \alpha \right\vert \frac{\varepsilon }{2}\right\} \right)=0$ for every $\varepsilon >0$ that $\left( \alpha T\right) x_{n}\stc\left(\alpha T\right) x$ in $Y$. Therefore, $\alpha T\in \mathcal{C}_{st}\left(U,V\right)$ holds.
\end{proof}

For a linear operator $S:U\rightarrow V$ between two vector spaces $U$ and $V$, its adjoint $S^*: V^*\rightarrow U^*$ is a linear operator between algebraic duals and is defined by $S^*(f)=f(Sx)$ for all $f\in V^*$ and $x \in U$, and it maintains the property $\|S^*\|=\|S\|$.
\begin{remark}
	As is well known, a linear operator $S:U \rightarrow V$ is bounded if and only if its adjoint operator $S^{\prime }: V^{\prime } \rightarrow U^{\prime }$ is bounded. Similarly, by considering Lemma \ref{equality}, a linear operator on normed spaces $S$ is $st$-bounded if and only if its adjoint operator $S^{\prime}$ is $st$-bounded. The situation for the $st$-continuous operators is the same by Proposition \ref{inclusion}.       	
\end{remark}

\section{Statistical compact-like operators}\label{Sec:4}
Consider an operator $S:U\to V$ between two normed spaces. If $S$ maps the closed unit ball $U$ of $U$ to a subset of $V$ that is relatively compact in terms of its norm i.e., $\overline{S(U)}$ is a compact subset of $V$, then $S$ is termed as a compact operator. In other words, $S$ is compact if, for every sequence $(x_n)$ in $U$ that is bounded in norm, there exists a subsequence $(x_{n_k})$ such that the sequence $S(x_{n_k})$ converges in $V$. With this in mind, we introduce the following concepts.
\begin{definition}
	A linear operator $S:U\to V$ is called {\em statistical compact} (or {\em $st$-compact}) operator if it sends $st$-bounded sequences to $st$-convergent sequences.
\end{definition}

We denote $\mathcal{K}_{st}(U,V)$ collection of all $st$-compact operators from $U$ to $V$.
\begin{remark}\
	\begin{enumerate}[(i)]
		\item Every $st$-compact operator is norm compact. But, the converse need not be true in general.
		\item Let $S:U \rightarrow V$ be a mapping between two Banach spaces. $S$ is {\em weakly compact} if every norm-bounded sequence $(x_n)$ in $U$ has a subsequence $(x_{k_n})$ such that the sequence $S(x_{k_n})$ converges weakly in $V$. Every $st$-compact operator is weakly compact because $st$-convergence implies weakly convergence on normed spaces; see \cite[Thm.2.3(i)]{BB}.
		\item It follows from \cite[Thm.2.3(iii)]{BB} that weakly convergence implies $st$-convergence. However, weakly compactness does not imply $st$-compactness on finite normed spaces in general.
	\end{enumerate}
\end{remark}

\begin{theorem}
	Every $st$-compact operator is $st$-bounded.
\end{theorem}

\begin{proof}
	Assume that $S\in \mathcal{K}_{st}(U,V)$. Take a $st$-bounded sequence $(x_n)$ in $U$. Then, $(S(x_n))$ is $st$-convergent in $V$, and so $S(x_n)$ is also $st$-bounded in $V$ because every $st$-convergent sequence is $st$-bounded. Hence, we get the desired result.
\end{proof}

Note that it is well known that an identity operator $I_U$ on a Banach space $U$ is compact if and only if $U$ is finite dimensional. However, this statement is not true for statistical compact operators because a $st$-bounded sequence need not be $st$-convergent.

\begin{theorem}
	Every $st$-compact operator is $st$-continuous.
\end{theorem}

\begin{proof}
	Assume $S$ is $st$-compact operator, but not $st$-continuous. This means there exists a sequence $(x_n)$ in $U$ such that it is $st$-convergent to $\theta$ in $U$, but $S(x_n)$ is not $st$-convergent to $\theta$ in $V$. Since $(x_n)$ is $st$-convergent to $\theta$, it is also statistically bounded. Now, let's consider the following two cases:
	
	\textit{Case 1:} Suppose that $S(x_n)$ is not statistically convergent to $\theta$ in $V$. In this case, for any positive real number $\varepsilon > 0$, there exists a set $A_\varepsilon \subseteq\mathbb{N}$ with $\delta(A_\varepsilon)\neq0$ such that $\|S(x_n)-\theta\| \geq \varepsilon$ holds for all $n \in A_\varepsilon$. On the other hand, since $(x_n)$ is statistically bounded, we can extract a subsequence $(x_{k_n})$ of $(x_n)$ such that $\delta(A_\varepsilon \cap \{k_n : n \in \mathbb{N}\})\neq 0$. This means there exists a set $B_\varepsilon \subseteq \mathbb{N}$ with $\delta(B_\varepsilon)\neq0$ such that for all $k \in B_\varepsilon$, $\|S(x_{k_n}) - \theta\| \geq \varepsilon$. Therefore, $(S(x_{k_n}))$ cannot be $st$-convergent to $\theta$ in $V$, which contradicts the assumption that every subsequence of a $st$-bounded sequence has a $st$-convergent subsequence image under $S$.
	
	\textit{Case 2:} Assume that $S(x_n)$ has a subsequence $(S(x_{k_n}))$ which is $st$-convergent to some $y \in V$, but different from $\theta$. Then, for each positive real number $\varepsilon > 0$, we have $\delta(\{n \in \mathbb{N}:\|S(x_{k_n}) - y\|\geq \varepsilon\}) = 0$. This implies that for any $\varepsilon > 0$, almost all terms of the subsequence $(S(x_{n_k}))$ are within a distance of $\varepsilon$ from $y$. However, since $(x_n)$ is $st$-convergent to $\theta$, for the same $\varepsilon > 0$ there exists a set $C_\varepsilon \subseteq \mathbb{N}$ with $\delta(C_\varepsilon) > 0$ such that for all $n \in C_\varepsilon$, we have $\|x_n - \theta\| < \varepsilon$. It follows from the linearity of $S$ that we have:
	$$
	\|S(x_n) - S(\theta)\| =\|S(x_n - \theta)\| \leq \|S\| \ \|x_n - \theta\|.
	$$
	Therefore, for all $n \in C_\varepsilon$, we get $\|S(x_n) - S(\theta)\|<\varepsilon\|S\|$. This contradicts the fact that for almost all $k \in \mathbb{N}$, $\|S(x_{n_k}) - y\| \geq \varepsilon$.
	
	In both cases, we arrive at a contradiction. Therefore, the original assumption that $S$ is not $st$-continuous must be false. This concludes the proof that every $st$-compact operator is $st$-continuous.
\end{proof}

\begin{theorem}
	If $(S_m)$ is a sequence of norm continuous $st$-compact operator from $U$ to $V$, where $V$ is Banach space, and $S_m\nc S$, then $S$ is $st$-compact. 
\end{theorem}

\begin{proof}
	Let $(x_n)$ denote a sequence in $U$ that is $st$-bounded. This implies that there exists a positive number $M$ such that the norm of $x_n$ is less than or equal to $M$ for almost all $n$ in the set of natural numbers. By utilizing a standard diagonal argument, we can establish the existence of a subsequence $(x_{n_k})_{n_k\in K}$ with $\delta(K)=1$. This subsequence has the property that for any natural number $m$, the image of $(x_{n_k})$ under the transformation $S_m$ norm converge to a limit $y_m$ in $V$. We aim to demonstrate that the sequence $(y_m)$ forms a $st$-Cauchy sequence in $V$.
	\begin{eqnarray*}
		\|y_m-y_j\|&=&\|y_m-S_mx_{n_k}+S_mx_{n_k}-S_jx_{n_k}+S_jx_{n_k}-y_j\|\\ &\leq& \|y_m-S_mx_{n_k}\|+\|S_mx_{n_k}-S_jx_{n_k}\|+\|S_jx_{n_k}-y_j\|.
	\end{eqnarray*}
	As $m \to \infty$ and $j \to \infty$, both the first and third terms in the last inequality converge to zero in norm. Since $S_m$ is a norm-continuous operator for all natural numbers $m$, we can conclude that:
	$$
	\|S_mx_{n_k}-S_jx_{n_k}\|\leq\|S_m-S_j\| \ \|x_{n_k}\|\leq\|S_m-S_j\|\ M
	$$
	for almost all $m,j$. As $(S_m)$ is a sequence that converges in norm, we have $\|S_m - S_j\| \to 0$ as $m$ and $j$ tend to infinity. Consequently, we obtain $\|y_m - y_j\| \to 0$ in $F$ as $m$ and $j$ approach infinity. Therefore, $(y_m)$ forms a $st$-Cauchy sequence. Since $V$ is complete with respect to the norm, there exists an element $y \in V$ such that $\|y_m - y\| \to 0$ in $F$ as $m$ tends to infinity. Thus, we have:
	\begin{eqnarray*}
		\|Sx_{n_k}-y\|&\leq& \|Sx_{n_k}-S_mx_{n_k}\|+\|S_mx_{n_k}-y_m\|+\|y_m-y\|\\ &\leq& \| S_m-S\|\ \|x_{n_k}\|+\|S_mx_{n_k}-y_m\|+\|y_m-y\|\\ &\leq& \|S_m-S\|\ M+\|S_mx_{n_k}-y_m\|+\|y_m-y\|.
	\end{eqnarray*}
	Fix $m\in N$ and let $k\to\infty$, then 
	$$
	\|Sx_{n_k}-y\|\leq\|S_m-S\| \ M+\|y_m-y\|.
	$$ 
	Since $m \in \mathbb{N}$ is arbitrary, we conclude that $|S(x_{n_k}) - y| \to 0$. Consequently, $S$ is $st$-compact.
\end{proof}

\begin{proposition}\label{leftandrightmultiplication}
	Let $R,T,S$ be linear operator on $U$.
	\begin{enumerate}
		\item[(i)] If $T$ is $st$-compact and $S$ is $st$-continuous, then $S\circ T$ is $st$-compact.
		\item[(ii)] If $T$ is $st$-compact and $R$ is $st$-bounded, then $T\circ R$ is $st$-compact.
	\end{enumerate}
\end{proposition}

\begin{proof}
	$(i)$ Suppose that $(x_n)$ is a $st$-bounded sequence in $U$. Since $T$ is $st$-compact. Then we have $T(x_n)\stc x$ for some $x\in U$. It follows from the $st$-continuity of $S$ that $S\big(T(x_n)\big)\stc S(x)$. Therefore, $S\circ T$ is $st$-compact.
	
	$(ii)$ Assume that $(x_n)$ is a $st$-bounded sequence in $U$. Since $R$ is $st$-bounded, then $R(x_n)$ is $st$-bounded. Now, the $st$-compactness of $T$ implies that $T\big(R(x_n)\big)\stc z$ for some $z\in U$. Therefore, $T\circ R$ is $st$-compact.
\end{proof}

\begin{proposition}\label{rankoneoperator}
	Let $S$ be a $st$-bounded finite rank operator between $U$ and $V$. Then $S$ is $st$-compact.
\end{proposition}

\begin{proof}
	We suppose that $S$ is given by $Sx=f(x)y_0$ for some $st$-bounded functional $f:U\to\mathbb{R}$ and $y_0\in V$. Let $(x_n)$ be a $st$-bounded sequence in $U$, then $f(x_n)$ is $st$-bounded in $\mathbb{R}$, and so there is a subsequence $x_{k_n}$ such that $f(x_{k_n})\to\lambda$ for some $\lambda\in\mathbb{R}$. Thus we have
	$$
	S(x_{k_n})-\lambda y_0=(f(x_{k_n})-\lambda)y_0\to0
	$$ 
	in $F$, and so we get $S(x_n)\stc 0$. Therefore, $S$ is $st$-compact.
\end{proof}

\begin{example}
	Consider a normed space $U$ and a linear functional $f:U\to\mathbb{R}$ that is not $st$-bounded. In this case, there exists a $st$-bounded sequence $(x_n)$ such that $\lvert f(x_n)\rvert\geq n$ for all $n\in \mathbb{N}$. Consequently, any rank one operator $S:U\to V$ defined by the rule $Sx=f(x)y_0$, where $0\ne y_0\in V$, is not $st$-compact.
\end{example}

\section{Statistical Completeness}\label{Sec:5}
In this section, we show the characterization and properties of statistical Cauchy sequences within normed spaces, aiming to provide clarity and establish a foundational understanding of their significance. While Cauchy sequences have a well-established definition in normed spaces, the concept of statistical Cauchy sequences lacks a universally accepted standard. To address this, we present the most frequently utilized definition for statistical Cauchy sequences.
\begin{definition}
	\cite{fridy}  A sequence $x=(x_n)$ is called \emph{statistical Cauchy sequence} in a normed space $U$ if, for any $\varepsilon>0$, there exists a number $n_\varepsilon\in\mathbb{N}$ such that  
	\begin{equation*}
		\lim\limits_{n\to\infty}\frac{1}{n}|\{k<n:\lVert x_k-x_{n_\varepsilon}\rVert\geq\varepsilon\}|=0,  
	\end{equation*}
	holds.
\end{definition}

It's commonly acknowledged that every Cauchy sequence is a statistical
Cauchy in normed spaces. However, the opposite statement is generally false,
even in Banach spaces. It can be seen from \cite[Theorem 1]{fridy} that a real valued sequence $(x_{n})$ statistically convergent if and only if it is statistical Cauchy. Unfortunately, this fact does not satisfy in normed spaces directly, and so we demonstrate the norm version of the fact piece by piece.
\begin{theorem} \label{st_cauchy covergent}
	Let $U$ be a normed space and $x=(x_{n})$ be a statistical convergent sequence in $U$. Then, $(x_{n})$ is statistical Cauchy sequence.
\end{theorem}

\begin{proof}
	Suppose that the sequence $x=(x_{n})$ is statistically convergent to $x$ in $U$. Then, for any $\varepsilon >0$, we have 
	\begin{equation*}
		\lim_{n\rightarrow \infty}\frac{1}{n}|\{k\leq n:\Vert x_{k}-x\Vert \geq\frac{\varepsilon }{2}\}|=0.
	\end{equation*}
	Hence, we can say that $\Vert x_n-x\Vert <\frac{\varepsilon }{2}$ hols for almost all $n$ in  $\mathbb{N}$. So, let us choose $n_\varepsilon\in \mathbbm{N}$ such that $\Vert x_{n_\varepsilon}-x\Vert <\frac{\varepsilon}{2}$ holds. Then, the following inequality
	$$
	\|x_n-x_{n_\varepsilon}\|\leq \|x_n-x\|+\|x_{n_\varepsilon}-x\|<\frac{\varepsilon}{2}+\frac{\varepsilon }{2}=\varepsilon.
	$$
	holds for almost all $n$ in $\mathbb{N}$. Therefore, the sequence $(x_{n})$ is statistically Cauchy in $U$.
\end{proof}

Neverthless, converse of this theorem is not true in any normed space as can
be seen from the example below.
\begin{example}
	Consider $c_{00}$, the space of real sequences that eventually terminate in zeros, is a normed space according to the norm $\left\Vert \left(x_{1},x_{2},\cdots\right) \right\Vert_{c_{00}}=\sup\left\{ \left\vert x_{n}\right\vert :n\in \mathbb{N}\right\}$. Let $(x_{n})$ be a sequence in $c_{00}$ such that $x_{n}=\left(x_{1},x_{2},\cdots,x_{n},0,0, \cdots\right) =\left( 1,\frac{1}{2},...,\frac{1}{n},0,0,\cdots\right)$ for all $n\in \mathbb{N}$. Take any $\varepsilon >0$, then there exist $n_{0}\in \mathbb{N}$ such that $\frac{1}{\varepsilon }<n_{0}$. So, we fix an arbitrary index $n_\varepsilon>n_0$. If we take $n\geq n_{0}$ holds, then we have 
	$$
	\|x_{n}-x_{n_{\varepsilon }}\|_{c_{00}}=\|(0,\cdots,0,\frac{-1}{n+1},\cdots,\frac{-1}{n_{\varepsilon }},0,0,\cdots...)\|_{c_{00}}=\frac{1}{n+1}<\frac{1}{n_{0}}<\varepsilon
	$$
	for the condition $n_\varepsilon>n$, or we have 
	$$
	\|x_{n}-x_{n_{\varepsilon}}\|_{c_{00}}=\|(0,\cdots,0,\frac{1}{n_\varepsilon+1},\cdots,\frac{1}{n},0,0,\cdots...)\|_{c_{00}}=\frac{1}{n_\varepsilon+1}<\frac{1}{n_{0}}<\varepsilon
	$$
	for the case $n>n_\varepsilon$. Hence, we obtain
	\begin{equation*}
		\delta \left( \left\{ k<n:||x_{k}-x_{n_{\varepsilon }}||\geq \varepsilon\right\} \right) \leq \delta \left( \left\{ 1,2,...,n_{0}\right\} \right) =0.
	\end{equation*}
	It follows that $(x_{n})$ is a statistical Cauchy sequence in $c_{00}$. 
	
	Now, choose an arbitrary element $x=\left( x_{1},x_{2},\cdots,x_{k},\cdots\right) \in c_{00}$. Then, there is  $j\in \mathbb{N}$ such that $x_{k}=0$ for every $k\geq j$, and so we have $\left\Vert x_{n}-x\right\Vert _{c_{00}}\geq \frac{1}{j}$ for every $j\leq n\in \mathbb{N}$. Taking $\varepsilon >0$ such that $0<\varepsilon <\frac{1}{j}$, we have 	\begin{equation*}
		\delta\left(\left\{k<n:||x_{k}-x||\geq\varepsilon \right\} \right)=1.
	\end{equation*}
	This means that $(x_{n})$ is not statistical convergent to any element $x=\left( x_{1},x_{2},...\right) \in c_{00}$.
\end{example}

Although being statistical Cauchy is not sufficient for convergence in general normed spaces, the situation is different in Banach spaces as we show in the next theorem. We refer the reader for an exposition of statistical convergence on Bnach spaces to \cite{Kolk}.
\begin{theorem}	\label{cauchy and convergence}
	Let $x=(x_{n})$ be a sequence in a Banach space $U$. Then, $(x_{n})$ is statistical Cauchy iff it is statistical convergent.
\end{theorem}

While this theorem might seem familiar, we were unable to locate any precise references to support it. Therefore, to ensure clarity, we present a succinct demonstration to aid comprehension.

\begin{proof}
	It is enough to show that every statistical Cauchy sequence in a Banach space $U$ is convergent due to Theorem \ref*{st_cauchy covergent}. Suppose that $(x_n)$ is a statistical Cauchy sequence in a Banach space $U$. Then, for each $\varepsilon>0$, there exists $n_{\varepsilon} \in \mathbb{N}$ as in the definition of statistical Cauchy such that  
	\begin{equation*}
		\lim_{n \to \infty} \frac{1}{n}|\{k<n : \|x_k - x_{n_{\varepsilon}}\| \geq\varepsilon\}|=0,
	\end{equation*}
	holds. Now, we construct a Cauchy subsequence in the usual sense $(x_{n_k})$ of $(x_n)$ as follows: pick $n_1=n_{\varepsilon_1}$ for some $\varepsilon_1>0$, and recursively, choose $n_{k+1}>n_k$ such that $n_{k+1}=n_{\varepsilon_{k+1}}$ for some $\varepsilon_{k+1} < \frac{\varepsilon_k}{2}$. We will show that the subsequence $(x_{n_k})$ of $(x_n)$ is a Cauchy sequence, in the usual sense. For any $\varepsilon > 0$, we can choose $m$ such that $\varepsilon_{m+1} < \frac{\varepsilon}{2}$ holds. Then, for any $p,q > m$, by the choice of $n_p$ and $n_q$, we have  
	\begin{equation*}
		\|x_{n_p}-x_{n_m}\|<\frac{\varepsilon}{2} \ \ \ \text{and} \ \ \ \|x_{n_q}-x_{n_m}\|<\frac{\varepsilon}{2}.  
	\end{equation*}
	Hence, by using the triangle inequality, we get  
	\begin{equation*}
		\|x_{n_p}-x_{n_q}\|\leq\|x_{n_p}-x_{n_m}\|+\|x_{n_q}-x_{n_m}\|< \frac{\varepsilon}{2}+\frac{\varepsilon}{2}=\varepsilon.  
	\end{equation*}
	Therefore, for any $\varepsilon > 0$, we found $n_0=n_m$ such that for all $n_p,n_q>n_0$, $\|x_{n_p} - x_{n_q}\| < \varepsilon$. This demonstrates that $(x_{n_k})$ is a Cauchy sequence in the usual sense in $U$.
	
	Next, we prove that the subsequence $(x_{n_k})$ is statistical convergent. It follows from the completeness of $U$ that the Cauchy subsequence $(x_{n_k})$ converges to an element $x \in U$. We prove that it is also statistically convergent to $x$. Fix any $\varepsilon>0$. Find some $m\in\mathbbm{N}$ such that $\varepsilon_{m+1}<\varepsilon/2$. This is possible due to the construction	of the subsequence. Consider the following set  
	\begin{equation*}
		K:=\{k<n_p:\lVert x_k - x\rVert \geq \varepsilon\}  
	\end{equation*}
	for any $p > n_m$. Split $K$ into two disjoint subsets $K_1$ and $K_2$ respectively as follows: 
	$$
	\{k<n_p:\lVert x_k - x_{n_m}\rVert \geq \varepsilon/2 \ \text{and}\ \lVert x_k - x\rVert \geq \varepsilon\}
	$$
	and 
	$$
	\{k < n_p : \lVert x_k - x_{n_m}\rVert <\varepsilon/2 \ \text{and} \ \lVert x_k-x\rVert \geq \varepsilon\}.
	$$
	Since $p > n_m$, we have $\lVert x_{n_p} - x_{n_m}\rVert < \varepsilon/2$. By the triangle inequality, $\lVert x_k - x_{n_p}\rVert \geq \varepsilon/2$ for any $k \in K_1$. Therefore, $K_1$ is a subset of the set $\{k < n:\lVert x_k - x_{n_m}\rVert \geq \varepsilon/2\}$. From the definition of a statistical Cauchy sequence, we know that the density of this set approaches $0$ as $n\rightarrow \infty$. Hence, the density of $K_1$ also approaches $0$ as $p \rightarrow \infty$. On the other hand, we show that $K_2$ is empty. Assume that there exists an element $k \in K_2$. Then $\lVert x_k-x_{n_m}\rVert < \varepsilon/2$ and $\lVert x_k - x\rVert \geq \varepsilon$.	By the triangle inequality, $\lVert x - x_{n_m}\rVert \geq \lVert x_k-x\rVert - \lVert x_k - x_{n_m}\rVert > \varepsilon/2$, which contradicts with the convergence of $(x_{n_k})$ to $x$. Therefore, $K_2$ must be empty.	Since $K_2$ is empty, $K=K_1$. We have established that the density of $K_1$ approaches $0$ as $p \rightarrow \infty$. This implies that the density of $K$ also approaches $0$ as $p \rightarrow \infty$, satisfying the definition of statistical convergence. As a result, we obtain that the subsequence $(x_{n_k})$ statistically converges to $x$.
\end{proof}

According to Theorem \ref{cauchy and convergence}, the sets of statistical Cauchy sequences and of statistically convergent sequences coincide in classical Banach spaces. Consequently, it seems natural to extent the concept of completeness in the statistical sense.

The normed space $\left( U,\left\Vert .\right\Vert \right) $ is called {\em statistical complete} (or {\em statistical Banach}) space provided that every statistical Cauchy sequence is statistical convergent to an element of $U$.

By considering Theorem \ref{cauchy and convergence}, it is obvious that every Banach space is a statistical Banach space. For converse, we can not find a suitable converse example.
\begin{question}
	Is there an example where a statistical Banach space is present but does not meet the criteria to be classified as a Banach space?
\end{question}

\end{document}